\documentclass[12pt]{article}
\usepackage[utf8]{inputenc}
\usepackage{color}
\usepackage[dvipsnames]{xcolor}
\usepackage{amsmath}
\usepackage{parskip}
\usepackage[showframe=false]{geometry}
\usepackage{changepage}
\usepackage{graphicx}
\usepackage{tikz}
\usepackage{ mathdots }
\usetikzlibrary{arrows.meta}
\usetikzlibrary{patterns}
\usepackage{ dsfont }
\usepackage{diagbox}
\usepackage[refpage]{nomencl}
\usepackage{amsthm}
\usepackage[compact]{titlesec}         % you need this package
\graphicspath{ {Images/} }
\usetikzlibrary{decorations.shapes}
\titlespacing{\section}{10pt}{10pt}{10pt} % this reduces space between (sub)sections to 0pt, for example
\AtBeginDocument{%                     % this will reduce spaces between parts (above and below) of texts within a (sub)section to 0pt, for example - like between an 'eqnarray' and text
  \setlength\abovedisplayskip{10pt}
  \setlength\belowdisplayskip{10pt}}
\usepackage{tikz}
\usetikzlibrary{positioning}
\usepackage{amssymb}
\usepackage{tikz}
\usetikzlibrary{patterns}
\usetikzlibrary{decorations.pathmorphing}
\usetikzlibrary{angles,quotes}
\tikzset{snake it/.style={decorate, decoration=snake}}
\usepackage{cancel}
\usepackage{caption}
\usepackage{subcaption}
\usepackage{ textcomp }
\makeindex
\usepackage{ gensymb }
\usepackage{float}
\usepackage{rotating}
\usepackage{tabularx}
\usepackage{delarray}

\newcommand{\tpmod}[1]{{\@displayfalse\pmod{#1}}}

\newcommand{\Sym}[1]{\mathrm{Sym}({#1})}

\usepackage{amsmath,xcolor}

\usepackage{mathtools}
\usepackage[shortlabels]{enumitem}

\usepackage[english]{babel}
\usepackage{soul}
\newtheorem{theorem}{Theorem}[section]
\newtheorem{corollary}[theorem]{Corollary}
\newtheorem{lemma}[theorem]{Lemma}
\newtheorem{definition}[theorem]{Definition}

\newtheorem{algorithm}[theorem]{Algorithm}
\newtheorem{example}[theorem]{Example}

\newtheorem{remark}[theorem]{Remark}

\theoremstyle{definition}

\usepackage{tikz}
\usepackage{siunitx}

\usepackage{graphicx}
\title{\bf{The Deletion Order and Coxeter Groups}}
\author{Robert Nicolaides and Peter Rowley}
\date{}

\begin{document}

\maketitle
\abstract The deletion order of a finitely generated Coxeter group $W$ is a total order on the elements which, as is proved, is a refinement of the Bruhat order. This order is applied in \cite{bruhatElnitskyTilings} to construct Elnitsky tilings for any finite Coxeter group.   Employing the deletion order, a corresponding normal form of an element $w$ of $W$ is defined which is shown to be the same as the normal form of $w$ using right-to-left lexicographic ordering.  Further results on the deletion order are obtained relating to the property of being Artinian and, when $W$ is finite, its interplay with the longest element of $W$.

\section{Introduction}\label{section intro}

The deletion order, which is a total order, has several interesting properties that pair nicely with Coxter groups. In the context of computation of finitely presented groups, it gets mentioned in the \textsc{Gap} user manual \cite{GAP}, and there it is called `the basic wreath order'.  A further mention is seen in  Section 2.1 of Sims \cite{sims_1994}. There appear to be no further manifestations of this order in the wider literature, and, to our knowledge, no investigations into its relationship with Coxeter groups.

Throughout this paper, we assume $W$ is a finitely generated Coxeter group whose fundamental (or simple) reflections are $S = \{s_1, \dots , s_n \}$. For $w$ an element of $W$, we let $\mathcal{R}(w)$ be the set of reduced words for $w$; those words over $S$, evaluating to $w$, using the minimum number of generators to do so. The length of these reduced words for $w$ is denoted by $\ell(w)$. The Bruhat order of a Coxeter group, $<_B$, is a celebrated partial order on $W$ that provides insights into the structure of reduced words of $W$. The definition of $<_B$, as well as some of its properties, are recalled in Section \ref{Deletionand Coxeter}. 

The definition of the deletion order on $W$, denoted by $<_{\Delta}^W$, starts from some chosen total order on $S$ and constructs a total order on the elements of $W$. 
The authors' original motivation for studying the deletion order came from a problem related to Elnitsky's tilings of polygons, related to Coxeter groups of type $\mathrm{A}$, $\mathrm{B}$ and $\mathrm{D}$ (see \cite{elnitsky}). The deletion order was used as part of the jigsaw in the solution for constructing analogous tilings for \emph{all} finite Coxeter groups (see \cite{Etilings},  \cite{bruhatElnitskyTilings}). An algorithm is presented in Section \ref{algorithm} which delivers a labelling of the Cayley graph $\mathrm{Cay}(W, S)$ when $W$ is finite, which we see in Theorem \ref{Theorem min min equal deletion} coincides with the deletion order on $W$. This type of labelling is central to the role of deletion orders in constructing tilings in \cite{bruhatElnitskyTilings}. We give a brief indication here of how this plays out. For a finite Coxeter system $(W, S)$ equipped with a total order $\ll$ on $W$, for $w \in W$, let 
$L_\ll(w) = |\{u \in W \,|\, u \ll w\}| + 1$. Then its induced Cayley embedding is given by $\varphi_\ll: W \hookrightarrow \Sym{|W|}$ where $$(L_\ll(u))\varphi_\ll(w) = L_\ll(uw^{-1}) \; \mbox{for all}  \; u,w \in W.$$  In \cite{bruhatElnitskyTilings}, it is shown that those $\ll$ that are a refinement of the Bruhat order, produce $\varphi_\ll$ that are so-called \emph{strong E-embeddings}. This type of embedding produces the desired Elnitsky tilings. Refinements of the Bruhat order are not difficult to construct, since $<_B$ is ranked by the length function of the Coxeter system, The deletion order, however, is highlighted as a total order which is \emph{not} ranked by length, with a simple construction that interacts nicely with other elements of Coxeter groups. 

As already noted, the deletion order's relationship with Coxeter groups looks to be of independent interest. We highlight two results presented here, by way of evidence. In Definition \ref{Definition NormalForm} we define $\mathrm{NF}_\Delta(w)$ as the minimal element of $\mathcal{R}(w)$ under $<_{\Delta}$.  Let $\mathrm{NF}_{\texttt{RLex}}(w)$ denote the normal form of $w$ using right-to-left lexicographical ordering (see Section \ref{Deletionand Coxeter} for a detailed description).  Then in Theorem \ref{Theorem DeltaNF = RLexNF}  we prove the following.

\begin{theorem}\label{Theorem 1DeltaNF = RLexNF} Suppose that $W$ is a finitely generated Coxeter group.
For all $w \in W$, $$\mathrm{NF}_\Delta(w) = \mathrm{NF}_{\texttt{RLex}}(w).$$
\end{theorem}

The tie-in with the Bruhat order is recorded in Corollary \ref{Lemma Bruhat Less than}. Our second result examines the consequences of the deletion order being an Artinian order.  We call $<_\Delta^W$ \textit{Artinian} if for all elements $w \in W$, the set $w^\downarrow := \{ u \in W | u <_\Delta^W w \}$ is finite.  Then Theorem \ref{corollary Artinian compact hyperbolic} gives the following characterization.

\begin{theorem} Suppose that $W$ is a finitely generated irreducible Coxeter group. Then $<_{\Delta}^W$ is Artinian for all choices of total order on $S$ if and only if $W$ is either a finite, affine, or compact hyperbolic group.
\end{theorem}

We now review the other contents of this paper. Section~\ref{deletionorder} defines the deletion order in the general setting of words on a finite alphabet while Section~\ref{auxilliary} establishes several properties of this order. Of these, we have Lemma \ref{LemmaSubwords} which looks at proper subwords and Lemmas \ref{Lemma append same gen comparison}, \ref{Lemma prepend same gen comparison}, \ref{Lemma split word on right} and \ref{lemma split words on the left} give a range of results on how the order and concatenation of words interact. In Section~\ref{Deletionand Coxeter} Coxeter groups enter the picture and, as already mentioned, the interaction between reduced words and the deletion order is examined. Our final section examines the interplay between the longest element of a Coxeter group of type A and B and the labelling of the Cayley graphs given by the successor algorithm in Section~\ref{algorithm}.

We thank the referee for their helpful comments.

\section{The Deletion Order}\label{deletionorder}
Let $\mathcal{A} = \{a_1, \ldots, a_n\}$ be a finite set, our \textit{alphabet}, equipped with a total ordering $<_\mathcal{A}$ implicitly defined so that $a_i<_\mathcal{A} a_j$ if and only if $i<j$. We call $\mathcal{A}^*$ our set of \textit{words} which consists of all finite sequences over $\mathcal{A}$. We will write our words as strings of successive letters without commas: $a_1a_2a_1$ and $a_1a_1$ being words, for example. We also include the empty word (being the unique sequence of length $0$) and denote this by $\mathrm{e}$. Given two words, we allow ourselves to \textit{concatenate} them to form a new word: intuitively this is done by placing one word after another so that concatenating $a_1a_2a_1$ and $a_1a_1$ gives $a_1a_2a_1a_1a_1$ (so order does matter here). If $w = a_{i_1}\ldots a_{i_k}$, then we define $w^{-1}$ to be the word given by the reversed sequence, $w^{-1} := a_{i_k}\ldots a_{i_1}$.

We now add the following terminology. Given some word $w = a_{i_1}\ldots a_{i_k}$ and selecting $a_j \in \mathcal{A}$, we can rewrite $w$ uniquely as the concatenation of the following words $ w = b_0a_jb_1\ldots a_jb_\ell$ where each $b_i$ is some (possibly empty) word that does not have any appearances of $a_j$. We define the \textit{$a_j$ deletion sequence} of $w$ to be the sequence of words consisting of $b_0,\ldots,b_\ell$. We denote this by writing $$\widehat{w}_j =  [b_0,\ldots,b_\ell].$$ In the case of no appearances of $a_j$ in $w$, we simply write $\widehat{w}_j=[w]$. Let $\lambda_j(w)$ be the length of $\widehat{w}_j$ as a sequence. This means that the number of occurrences of $a_j$ in $w$ is equal to $\lambda_j(w)-1$.

The cases that some $b_i$ are empty (that is $b_i = \mathrm{e}$) are covered as follows: 
\begin{enumerate}[$(i)$]
    \item $i=0$ and $a_{i_1}=a_j$;
    \item $i=\ell$ and $a_{i_k}=a_j$; and
    \item $0<i<\ell$ and the $i^{th}$ and $(i+1)^{th}$ appearance of $a_j$ in $w$ (when reading $w$ from left to right) are adjacent.
\end{enumerate}

\begin{example}\label{ExampleDeletionSequences}
As an example we take $\mathcal{A} = \{a_1,a_2,a_3\}$ and compute the deletion sequences of $u = a_1a_2a_3a_1a_2a_2$, $v = a_1a_2a_3a_2a_1 $ and $w = a_3a_3$.
\begin{table}[H]
    \centering
    \begin{tabular}{l l l}
        $\widehat{u}_3=[a_1a_2,a_1a_2a_2]$ & $\widehat{u}_2 = [a_1,a_3a_1,\mathrm{e},\mathrm{e}]$ &$\widehat{u}_1 = [\mathrm{e},a_2a_3,a_2a_2]$ \\
        $\widehat{v}_3=[a_1a_2,a_2a_1]$ & $\widehat{v}_2 = [a_1,a_3,a_1]$ &$\widehat{v}_1 = [\mathrm{e},a_2a_3a_2,\mathrm{e}]$ \\
        $\widehat{w}_3=[\mathrm{e},\mathrm{e},\mathrm{e}]$ & $\widehat{w}_2 = [a_3a_3]$ &$\widehat{w}_1 = [a_3a_3]$ \\
    \end{tabular}
    \caption{The deletion sequences for words $u,v$ and $w$.}
    \label{tab:my_label}
\end{table}

\end{example}

\begin{definition}[The Deletion Order]\label{Definition Deletion order}
Given our alphabet $\mathcal{A} = \{a_1,\ldots, a_n\}$, we define our order inductively on $\mathcal{A}_i = \{a_1,\ldots, a_i\}$ for $i = 1,\ldots, n$. 

For the base case, suppose we have two distinct words $u,v \in \mathcal{A}_1^*$. Necessarily such words are of the form 
    $u = \underset{p \,\,\text{times}}{a_1\ldots a_1}$ and $v = \underset{q \,\,\text{times}}{a_1\ldots a_1}$ for some non-negative integers $p$ and $q$ (allowing implicitly for the appearance of the empty word). We say $u <_\Delta ^1 v$ if and only if $p < q$.
    
For $i > 1$, we define $<_\Delta ^{i}$ on words in $\mathcal{A}^*_i$ so that for distinct words $u$ and $v$, we first compute their $a_i$-deletion sequences:
\begin{align*}
    \widehat{u}_i &= [b_0,\ldots, b_k]\\
    \widehat{v}_i &= [c_0,\ldots, c_\ell].
\end{align*}
If $k \ne \ell$, and $k < \ell$, say, then this immediately implies $u <_\Delta ^i v$, and similarly, $k > \ell$ implies $v <_\Delta ^i u$. Otherwise, when $k = \ell$, we select the least $j$ such that $b_j \ne c_j$ and recursively compare these words by $<_\Delta ^{i-1}$; we say $b_j <_\Delta ^{i-1} c_j$ implies $u <_\Delta ^{i} v $ and conversely, $c_j <_\Delta ^{i-1} b_j$ implies $v <_\Delta ^{i} u$. Note that this is well-defined since, by construction, $b_j$ and $c_j$ are words contained in $\mathcal{A}^*_{i-1}$. 

We then put $<_{\Delta} = <_\Delta ^{n}$, the deletion order on $\mathcal{A}^*$.
\end{definition}

\begin{remark}
\begin{enumerate}[$(i)$]
\item We point out that comparing to see which is the larger integer of $k$ and $\ell$ is equivalent to comparing $\lambda_i(u)$ and $\lambda_i(v)$.
\item In the standard terminology within the realm of word orders, $<_\Delta ^{n}$ is the \textit{short-lexicographic} extension of $<_\Delta ^{n-1}$ and therefore a total order for all $n$ by induction.
\item We suspect that $<_\Delta$ can be characterized in terms of Prim's greedy algorithm (\cite{prims}) on the Cayley graph for the free monoid generated by $\mathcal{A}$ although we do not explore that here.
\end{enumerate}
\end{remark}
\begin{example}\label{ExampleDeletionOrder}
We revisit the words $u = a_1a_2a_3a_1a_2a_2$, $v = a_1a_2a_3a_2a_1 $ and $w = a_3a_3$ over $\mathcal{A} = \{a_1,a_2,a_3\}$ from Example \ref{ExampleDeletionSequences}.

Comparing $u$ and $v$ by $<_\Delta  = <_\Delta ^3$, we first consider the deletion sequences with respect to $a_3$:
\begin{align*}
    \widehat{u}_3 &= [a_1a_2,a_1a_2a_2]\\
    \widehat{v}_3 &= [a_1a_2,a_2a_1].
\end{align*}
These are sequences of the same length and so we isolate the leftmost respective entries in the sequences that differ; the first entry in both sequences is $a_1a_2$ and so we move on to the second entries where we see the different words $b:=a_1a_2a_2$ and $c:=a_2a_1$ for $u$ and $v$ respectively.

We go on to compare $b$ and $c$ by $<_\Delta ^2$. Again, we compute their deletion sequences, this time with respect to $a_2$:
\begin{align*}
    \widehat{b}_2 &= [a_1,\mathrm{e},\mathrm{e}]\\
    \widehat{c}_2 &= [\mathrm{e},a_1].
\end{align*}
In this case, we see that $\widehat{b}_2$ has the longer sequence, hence we determine $c <_\Delta ^2 b$ and therefore $v <_\Delta  u$.

To compare $u$ and $w$ we immediately see that $u <_\Delta  w$ since $w$ has more appearances of $a_3$ compared to $v$. Therefore $v <_\Delta  u <_\Delta  w$.
\end{example}

For convenience, if $w \in \mathcal{A}^*$ and $a_j \in \mathcal{A} = \{a_1,\ldots,a_n\}$ with deletion sequence $\widehat{w}_j = [b_0,\ldots,b_k]$, we set $\widehat{w}_j[i] = b_i$ for all $i = 0,\ldots,k = \lambda_n(w)-1$.

Due to its later importance, for each word $w \in \mathcal{A}^*$, and $k = 1,\ldots,n$, we will define a new word $\delta_k(w)$. To do this, we first define the auxiliary terms $\overline{\delta_k}(w)$ and ${\tau}(w)$. We write $$
\overline{\delta_k}(w) = \begin{dcases} \mathrm{e} & \text{if} \,\, \lambda_k(w)=1\\
\widehat{w}_k[0]a_k \widehat{w}_k[1]a_k\ldots a_k\widehat{w}_k[\lambda_k(w)-1]a_k & \text{otherwise.}  \end{dcases} 
$$ and $\tau(w) = \widehat{w}_k[\lambda_k(w)]$. Now we define $\delta_n(w) = \overline{\delta_{n}}(w)$ and $$\hspace{2cm}
\delta_k(w) =  ({\tau_{n}} \circ \ldots \circ {\tau_{k + 1}} \circ \overline{\delta_{n-k}})(w)$$ for all $k < n$.
Here, $\circ$ denotes function composition evaluated from left to right. This means $\delta_{n}(w)$ is the segment of the word $w$ up to and including its last appearance of $a_n$ when reading from the left. And $\delta_{n-1}(w)$ is the initial segment of the remaining letters to the right of $\delta_{n}(w)$ in $w$, up to and including its last appearance of $a_{n-1}$ and so on; iteratively isolating the last appearance of a given letter from the remaining word. Note that $\delta_i(w) \in \mathcal{A}_i^*$.

\begin{example}\label{Exmaple deltas}
Continuing with our familiar examples of $u = a_1a_2a_3a_1a_2a_2$, $v = a_1a_2a_3a_2a_1 $ and $w = a_3a_3$ in $\mathcal{A} = \{a_1,a_2,a_3\}$ we see that 
\begin{table}[H]
    \centering
    \begin{tabular}{l l l}
    $\delta_3(u) = a_1a_2a_3$ & $\delta_2(u) = a_1a_2a_2$ & $\delta_1(u) = \mathrm{e}$\\
    $\delta_3(v) = a_1a_2a_3$ & $\delta_2(v) = a_2$ & $\delta_1(v) = a_1$ \\ 
    $\delta_3(w) = a_3a_3$ & $\delta_2(w) = \mathrm{e}$ & $\delta_1(w) = \mathrm{e}$ \\    
    \end{tabular}
    \label{LambdaTable}
\end{table}
\end{example}

\section{Auxiliary Lemmas on General Words}\label{auxilliary}

Given $u,v \in \mathcal{A}^*$, we call $u$ a subword of $v$ if it corresponds to a subsequence of $v$ as finite sequences. Intuitively, this means we can obtain $u$ from $v$ by making several \textit{deletions} from $v$. For example, $a_1a_2a_1$ is a subword of $a_3a_1a_1a_2a_1a_3$ by making the deletions highlighted by $\widehat{a_3}\widehat{a_1}a_1a_2a_1\widehat{a_3}$ or $\widehat{a_3}a_1\widehat{a_1}a_2a_1\widehat{a_3}$. This sets the scene for our first lemma.

\begin{lemma}\label{LemmaSubwords}
If $u$ is a (proper) subword of $v$, then $u <_\Delta  v$.
\end{lemma}
\begin{proof}
Our strategy is to show that performing any single deletion produces a lesser word in $<_\Delta $, the rest following by transitivity. 
We proceed by induction on the least $i$ such that $v \in \mathcal{A}_i^*$, by showing that if we make any single deletion in $v$, the resulting word, $u$, is lesser in $<_\Delta $. For the case when $v \in A_1^*$, we necessarily have $v = \underbrace{a_1\ldots a_1}_{p \,\,\text{times}}$ and $u = \underbrace{a_1\ldots a_1}_{p-1 \,\,\text{times}}$ for some $p>0$. By definition of $<_\Delta ^1$ we immediately deduce that $u<_\Delta  v$. 

For the induction step, consider $v \in \mathcal{A}_i^*$, with $i>1$, and the consequence of making a single deletion to form the word $u$. Either, the letter we delete is $a_i$ or not.  In the case that it is indeed $a_i$, then $u$ has one less appearance of $a_i$ than $v$ and hence $\lambda_i({u}) = \lambda_i(v)-1 < \lambda_i({v})$. As a direct consequence of the construction of $<_\Delta $, $u <_\Delta ^i v$ in this case, whence $u <_\Delta  v$.

Write $v= b_0a_i\ldots a_ib_k$ as in the deletion sequence $\widehat{v}_i$. If the letter we delete to form $u$ is not $a_i$, then the deletion occurred within some (non-empty) $b_j \in \widehat{v}_i$. Call this resulting new word $\widehat{b_j}$, then 
\begin{align*}
    \widehat{v}_i &= [b_0,\ldots,b_j,\ldots,b_k] \\
    \widehat{u}_i &= [b_0,\ldots,\widehat{b_j},\ldots,b_k]
\end{align*}
Comparing $u$ and $v$ by $<_\Delta ^i$ reduces to comparing $b_j$ and $\widehat{b_j}$ by $<_\Delta ^{i-1}$ whence we are done by induction since $b_j,\widehat{b_j} \in \mathcal{A}_{i-1}^*$.
\end{proof}

For all $w \in \mathcal{A}^*$, let $\tau_n(w) = \widehat{w}_n[\lambda_n(w)]$. Consequently, we have $w = \delta_n(w)\tau_n(w)$. Lemma \ref{Lemma delta dichotomy} shows how this decomposition interacts with the deletion order.

\begin{lemma}\label{Lemma delta dichotomy}
Suppose that $u,v \in \mathcal{A}^*$. Then $u <_\Delta  v$ if and only if either
\begin{enumerate}[$(i)$]
    \item $\delta_n(u) <_\Delta  \delta_n(v)$; or
    \item $\delta_n(u) = \delta_n(v)$ and $\tau_n(u) <_\Delta^{n-1} \tau_n(v)$.
\end{enumerate}
\end{lemma}
\begin{proof}
Suppose that $\lambda_n(u) \ne \lambda_n(v)$.
Note that $\lambda_n(\delta_n(u)) = \lambda_n(u)$ and $\lambda_n(\delta_n(v)) = \lambda_n(v)$. So $\lambda_n(u)<\lambda_n(v)$ if and only if $\delta_n(u) <_\Delta  \delta_n(v)$ in this case.

Now assume that $\lambda_n(u) = \lambda_n(v)= \ell$. Writing 
\begin{align*}
    \widehat{u}_n &= [b_0,\ldots,b_{\ell -2},b_{\ell -1}] &&\text{and}\\
    \widehat{v}_n &= [c_0,\ldots,c_{\ell-2},c_{\ell -1}]\\\intertext{ gives }
     \widehat{\delta_n(u)}_n &= [b_0,\ldots,b_{\ell -2},e] &&\text{and}\\
    \widehat{\delta_n(v)}_n &= [c_0,\ldots,c_{\ell-2},e].    
 \end{align*}
 Further suppose, without loss of generality, that $u<_\Delta  v$. Directly from the definition of the deletion order, this is equivalent to the existence of some least $i$ within the range $0 \le i \le \ell-1$ such that $b_j = c_j$ for all $j<i$ but $b_i <_\Delta^{n-1} c_i$. Note that the first $\ell-1$ entries of the pairs $\widehat{u}_n$ and $\widehat{\delta_n(u)}_n$, and $\widehat{v}_n$ and $\widehat{\delta_n(v)}_n$, are respectively equal. So if $i < \ell-1$, then $b_i <_\Delta^{n-1} c_i$ implies that $u <_\Delta^{n} v$ if and only if $\delta_n(u) <_\Delta^{n} \delta_n(v)$. Finally, if $i = \ell -1$, then $\delta_n(u) = \delta_n(v)$ and so, using the definition of the deletion order, $u <_\Delta  w$ exactly when $\tau_n(u) <_\Delta^{n-1} \tau_n(v)$.
\end{proof}

\begin{lemma}\label{Lemma append same gen comparison}
Suppose $u$ and $v$ are distinct words in $\mathcal{A}^*$ and $a \in \mathcal{A}$. Then
\begin{align*}
    ua &<_{\Delta} va\\\intertext{if and only if}
    u &<_{\Delta} v.
\end{align*}
\end{lemma}
\begin{proof}
Let
\begin{align*}
    \widehat{u}_n &= [b_0,\ldots,b_{\ell -2},b_{\ell -1}] \quad \text{and}\\
    \widehat{v}_n &= [c_0,\ldots,c_{k-2},c_{k -1}].
\end{align*}
Then
\begin{align*}
    {\widehat{(ua)}}_n &= 
    \begin{cases} 
    [b_0,\ldots ,b_{\ell-1}a] \text{\,\,\quad if $a \ne a_n$, }\\
    [b_0,\ldots ,b_{\ell-1},e] \text{\quad if $a = a_n$,}\\
    \end{cases} \text{ and }\\
{\widehat{(va)}_n} &=    
    \begin{cases} 
    [c_0,\ldots, c_{\ell-1}a] \text{\,\,\quad if $a \ne a_n$, }\\
    [c_0,\ldots, c_{\ell-1}, e] \text{\quad if $a = a_n$. }\\
    \end{cases}
\end{align*}
Note that $\lambda_n(ua) = \lambda_n(u) + \epsilon$ and $\lambda_n(va) = \lambda_n(v) + \epsilon$ where $\epsilon = 0$ or $1$. So 

$$\lambda_n(u)<\lambda_n(v) \mbox{ if and only if} \; \lambda_n(ua)<\lambda_n(va).$$

It remains to check  the case $\lambda_n(ua) = \lambda_n(va)$.

If $a = a_n$, then we compare ${\widehat{(ua)}}_n =[b_0,\ldots ,b_{\ell-1},e]$ and ${\widehat{(va)}_n}=[c_0,\ldots, c_{\ell-1}, e]$ via $<_\Delta^{n-1}$. But this is evidently equivalent to comparing ${\widehat{u}}_n =[b_0,\ldots ,b_{\ell-1}]$ and ${\widehat{v}_n} =[c_0,\ldots, c_{\ell-1}]$ via $<_\Delta^{n-1}$.

If $a \ne a_n$, then we compare ${\widehat{(ua)}}_n =[b_0,\ldots ,b_{\ell-1}a]$ and ${\widehat{(wa)}_n}=[c_0,\ldots, c_{\ell-1}a]$ via $<_\Delta$. There exists some least $0 \le i < \ell -1$ for which $b_i \ne c_i$. If $i < \ell-1$, then since the first $\ell-2$ entries of ${\widehat{(ua)}}_n$ and ${\widehat{u}}_n$, and ${\widehat{(wa)}}_n$ and ${\widehat{w}}_n$ are respectively equal, $ua<_\Delta wa$ if and only $u<_\Delta w$.  If $i = \ell -1$, then the statement follows by induction on $|\mathcal{A}|$.
\end{proof}

We now give the `left-handed' version of Lemma \ref{Lemma append same gen comparison}, omitting the proof due to it being very similar.
\begin{lemma}\label{Lemma prepend same gen comparison}
Suppose $u$ and $v$ are distinct words in $\mathcal{A}^*$ and $a \in \mathcal{A}$. Then
    $ au <_{\Delta} av $ {if and only if} 
     $u <_{\Delta} v$.
\end{lemma}

Together, Lemma \ref{Lemma append same gen comparison} and Lemma \ref{Lemma prepend same gen comparison} show that the deletion order is what is known as a \textit{reductive order} in the parlance of term-rewriting. This means that we can discount the largest prefix and suffix common to both words when comparing in the deletion order.

\begin{lemma}\label{Lemma split word on right}
Let $u,v \in \mathcal{A}^*$ and $a_i,a_j \in \mathcal{A}$ be such that the following inequalities hold: \begin{align*}
    u<_\Delta ua_i, \quad
    v<_\Delta ua_i, \quad
    u<_\Delta va_j, \quad
    v<_\Delta va_j.
\end{align*} Then $ua_i <_\Delta  va_j$ if and only if either
\begin{enumerate}[$(i)$]
    \item $a_i <_{\Delta} a_j$; or
    \item $a_i = a_j$ and  $u <_\Delta  v$.
\end{enumerate}

\end{lemma}
\begin{proof}
If $a_i = a_j$, then statement (ii) follows by Lemma \ref{Lemma append same gen comparison}.

Suppose $a_i \ne a_j$. 
Write
\begin{align*}
    \widehat{u}_n &= [b_0,\ldots,b_{\ell -2},b_{\ell -1}] \quad \text{and}\\
    \widehat{v}_n &= [c_0,\ldots,c_{k-2},c_{k -1}].
    \\\intertext{ Consequently, }
    \widehat{(ua_i)}_n &= 
    \begin{cases} 
    [b_0,\ldots,b_{\ell-1}a_i]\text{\quad if $a_i \ne a_n$, }\\
    [b_0,\ldots,b_{\ell-1},e] \text{\quad if $a_i = a_n$, }
    \end{cases} \text{ and }\\
    \widehat{(va_j)}_n &= 
    \begin{cases} 
    [c_0,\ldots,c_{k-1}a_j]\text{\quad if $a_j \ne a_n$, }\\
    [c_0,\ldots,c_{k-1},e] \text{\quad if $a_j = a_n$. }
    \end{cases} 
\end{align*}
Suppose $a_j = a_n$, so $a_i < a_j$.  By Lemma \ref{Lemma delta dichotomy}, comparing $ua_i$ and  $va_n$ in $<^n_\Delta$ is equivalent to first comparing $\delta_n(ua_i)$ and $ \delta_n(va_n)$ by $<^{n}_\Delta$ and then comparing $\tau_n(ua_i) $ and $ \tau_n(va_n)$ by $<^{n-1}_\Delta$. Note that $\delta_n(ua_i)=\delta_n(u)$ and $\delta_n(va_n)=va_n$. This fact presents several contradictions that we tackle case by case.

If $\delta_n(va_n) <^n_\Delta \delta_n(ua_i)$,  then $\delta_n(va_n) <^n_\Delta \delta_n(u)$, whence $va_n <_\Delta u$, a contradiction. 

If $\delta_n(ua_i) = \delta_n(va_n) = va_n$, then either $va_n$ must be a proper subword of $ua_i$, leading to the contradiction $va_n <_\Delta u$ again, by Lemma \ref{LemmaSubwords}. 

Finally, if $\delta_n(ua_i) <^n_\Delta \delta_n(va_n)$, then the result $ua_i<^n_\Delta va_n$ by Lemma \ref{Lemma delta dichotomy} applies directly. 

If neither $a_i$ nor $a_j$ is $a_n$, then the result follows by induction on $|\mathcal{A}|$.
\end{proof}

Now we consider the `left-handed' version of Lemma \ref{Lemma split word on right}.

\begin{definition}\label{definition reverse delta vector}
Let $w = a_{i_1}\ldots a_{i_k}\in \mathcal{A}^*$. We define $$\alpha(w) = [\lambda_n(\delta_{n}(w^{-1})),\lambda_{n-1}(\delta_{n-1}(w^{-1})),\ldots,\lambda_{1}(\delta_{1}(w^{-1}))].$$ 
\end{definition}
So the first entry of $\alpha$ counts the number of appearances of $a_n$ in $w$.  The second entry counts the number of appearances of $a_{n-1}$ in the subword $a_{i_1}\ldots a_{i_j}$ where $a_{i_{j+1}}$ is the first appearance of $a_n$ in $w$ (if such an appearance exists) and so on in this fashion. 

\begin{example} Suppose that $\mathcal{A} = \{a_1,a_2,a_3 \}$, with $u = a_1a_2a_1a_3a_3a_1a_2$ and $v=a_1a_1a_3a_3a_1a_2$. Then $\alpha(u) = [2,1,1]$ and $\alpha(v) = [2,0,2]$.
\end{example}
In the following lemma, we will want to compare the lists produced by $\alpha$ using the usual lexicographic order on integers. We denote this $<_\texttt{Lex}$ and recall the definition of how to use it to compare two lists: compare the leftmost entry in the lists that have different values; we say that the list with the lesser value is lesser with respect to $<_\texttt{Lex}$. 

\begin{lemma}\label{lemma split words on the left}
Let $u,v \in \mathcal{A}^*$ %(such that we don't have identical letters touching)
and $a_i, a_j \in \mathcal{A}$ be such that $a_i <_\mathcal{A} a_j$ and the following inequalities hold: 
\begin{align*}
    u <_\Delta a_iu, \quad
    v <_\Delta a_iu, \quad
    u <_\Delta a_jv, \quad
    v <_\Delta a_jv.
\end{align*} Then the following are equivalent: 
\begin{enumerate}[(i)]
    \item $a_iu <_\Delta  a_jv$;
    \item $\lambda_j(\delta_{j}((a_iu)^{-1}) = \lambda_j(\delta_{j}((a_jv)^{-1})) - 1;$
    \item $\alpha(a_iu) <_\texttt{Lex} \alpha(a_jv)$.

\end{enumerate} 

\end{lemma}
\begin{proof}
    We start with $(i)$ implies $(ii)$.
    Supposing that $a_iu <_\Delta  a_jv$, we walk through the procedure of comparing the words step-by-step. We start by comparing $\lambda_n(a_iu)$ and $ \lambda_n(a_jv)$. Now there are two cases to consider: $j = n$ or $j < n$.

If $j = n$, we know $\lambda_n(\delta_n(u^{-1})) = \lambda_n(u) = \lambda_n(u) = \lambda_n(a_iu)$ and $\lambda_n(v) = \lambda_n(\delta_n(v^{-1})) = \lambda_n(a_jv) - 1$. The constraints from our hypothesis, $v <_\Delta a_iu$ and $u <_\Delta a_jv$, further imply that $\lambda_n(v) \le \lambda_n(a_iu)$ and $\lambda_n(u) \le \lambda_n(a_jv)$. From this, we can deduce  $\lambda_n(\delta_n(v^{-1})) \le \lambda_n(\delta_n(u^{-1})) \le \lambda_n(\delta_n(v^{-1})) + 1$. But notice $\lambda_n(\delta_n(u^{-1})) = \lambda_n(\delta_n(v^{-1})) + 1$ would contradict $u <_\Delta a_jv$, since the next step in determining $a_iu <_{\Delta} a_jv$ would be to compare $\widehat{a_iu}_n[0]$ (a non-empty word since it contains $a_i$) to $\widehat{a_jv}_n[0]$ (the empty word). Hence $\lambda_n(\delta_n(u^{-1})) = \lambda_n(\delta_n(v^{-1})) = \lambda_n(\delta_n((a_jv)^{-1})) - 1$

        If $j < n$, then $\lambda_n(u) = \lambda_n(v) = \lambda_n(a_iu) = \lambda_n(a_jv)$, hence we must now compare $\widehat{a_iu}_n[0]$ and $\widehat{a_jv}_n[0]$ via $<_{\Delta_{n-1}}$, where the result follows by induction on $|\mathcal{A}|$.

    \par Next, we prove $(ii)$ implies $(iii)$. Again, we first consider the case that $a_j = a_n$. In this case, to determine $\alpha(a_iu) <_\texttt{Lex} \alpha(a_jv)$, we start by comparing their leftmost entries, which are $\lambda_j(\delta_{j}((a_iu)^{-1})$ and $ \lambda_j(\delta_{j}((a_jv)^{-1}))$ respectively. But since $\lambda_j(\delta_{j}((a_iu)^{-1}) = \lambda_j(\delta_{j}((a_jv)^{-1})) - 1$, the result follows. 
    
    In the case that $a_j < a_n$, for each $j < k \le n$, we must have $\lambda_k(\delta(u^{-1})) = \lambda_k(\delta((a_iu)^{-1}))$ and $\lambda_k(\delta(v^{-1})) = \lambda_k(\delta((a_jv)^{-1}))$. Our hypothesis that $v <_\Delta a_iu$ and $u <_\Delta a_jv$, adds to our knowledge that $\lambda_k(\delta(v^{-1})) = \lambda_k(\delta((a_iu)^{-1}))$ and $\lambda_k(\delta(u^{-1})) = \lambda_k(\delta((a_jv)^{-1}))$. Hence, when comparing $\alpha(a_iu)$ and $ \alpha(a_jv)$ by $<_\texttt{Lex}$, the first $k$ elements to the left are equal. The next comparison shows us that $\lambda_j(\delta_{j}((a_iu)^{-1}) = \lambda_j(\delta_{j}((a_jv)^{-1})) - 1$, whence the result follows.

    \par Finally we prove $(iii)$ implies $(i)$. First, we consider the case that $a_j = a_n$. We perform the step-by-step comparison of $a_iu$ and  $a_jv$ by $<_\Delta$ given that $\alpha(a_iu) <_\texttt{Lex} \alpha(a_jv)$. We compare $\lambda_n(\delta_{j}((a_iu)^{-1}) = \lambda_n(a_iu)$ and $ \lambda_n(\delta_{j}((a_jv)^{-1})) = \lambda_n(a_jv)$. We consider the three cases of how $\lambda_n(a_iu)$ and $ \lambda_n(a_jv)$ compare. If $\lambda_n(a_iu) > \lambda_n(a_jv)$ this would contradict $\alpha(a_iu) <_\texttt{Lex} \alpha(a_jv)$. If $\lambda_n(a_iu) < \lambda_n(a_jv)$ then $a_iu <_\Delta a_jv$. Next, if $\lambda_n(a_iu) = \lambda_n(a_jv)$, then 
    we next compare $\widehat{a_iu}_n[0] = \delta_{n-1}((a_iu)^{-1})$ and $\widehat{a_jv}_n[0] = \delta_{j}((a_jv)^{-1})$ via $<_{\Delta_{n-1}}$. But  $\delta_{j}((a_jv)^{-1})$ is the empty word, whereas $\delta_{n-1}((a_iu)^{-1}$ is not

    If $a_j < a_n$, then we can deduce given our hypotheses that $\lambda_n(a_iu) = \lambda_n(a_jv)$ whence the result follows by induction on $|\mathcal{A}|$.
\end{proof}

\section{Deletion order and Coxeter groups}\label{Deletionand Coxeter}
Let $(W, S)$ be a finitely generated Coxeter group of rank $n$. Also we assume that $S=\{s_1,\ldots,s_n\}$ with $<_S$ an order on $S$ implicitly given by $s_i <_S s_j$ if and only if $i < j$.  For $J$  a subset of $S$, $W_J$ denotes the standard parabolic subgroup $\langle J \rangle$.
With this setup, we show how one can form the deletion order for words over the alphabet $S$.

\begin{definition}\label{Definition NormalForm}
Let $w \in W$ and recall $\mathcal{R}(w)$ denotes the set of reduced words of $w$.

We define the normal form of $w$ with respect to $<_\Delta$ to be that word in $\mathcal{R}(w)$ which is minimal with respect to the deletion order over $S$. We denote this word by $\mathrm{NF}_\Delta(w)$.
\end{definition}
This normal form is well defined since $\mathcal{R}(w)$ is always finite (even if $W$ were not) and the deletion order is a total order.

Now we introduce a total ordering on the elements of $W$.

\begin{definition}\label{Definition Deletion Order on W}
For distinct $w_1,w_2 \in W$, we say $w_1 <_\Delta^W w_2$ whenever the normal form of $w_1$ is less than that of $w_2$ in the deletion order. That is, $$w_1 <_\Delta^W w_2 \,\,\,\text{if and only if}\,\,\, \mathrm{NF}_\Delta(w_1) <_\Delta  \mathrm{NF}_\Delta(w_2). $$  
\end{definition}
    There are a variety of normal forms which appear in the literature. For each $w \in W$ we define $\mathrm{NF}_{\texttt{LLex}}(w)$ to be its lexicographically least reduced word with respect to $<_S$. This is computed by respectively comparing their corresponding letters (in this case  of $S$) of reduced words of $w$ from left to right; the word with the first occurrence of the larger letter (with respect to $<_S$) produces the larger word. Similarly, the reverse lexicographic order is computed identically but in a right-to-left fashion. We denote its normal form by $\mathrm{NF}_{\texttt{RLex}}(w)$. 
\begin{theorem}\label{Theorem DeltaNF = RLexNF}
For all $w \in W$, $$\mathrm{NF}_\Delta(w) = \mathrm{NF}_{\texttt{RLex}}(w).$$
\end{theorem}
\begin{proof}
We proceed by induction on $\ell(w)$. 

The base case concerns the identity element, $1$. Since the only reduced word for the identity is the empty word, necessarily we have $\mathrm{NF}_\Delta(1) = \mathrm{NF}_{\texttt{RLex}}(1)$.

Let $k\ge 0$, then for the induction step, we suppose our statement holds for all $u \in W$ of length at most $k$. Let $w \in W$ be such that $\ell(w) = k+1$. Now we consider each $s \in S$, such that $\ell(ws) = \ell(w)-1$ (if no such $s$ exists, then $w = 1$). Recall that such an $s$ exists if and only if $s$ may appear as the rightmost generator in some reduced word for $w$, by Lemma 2.4.3 of  \cite{COCG}. 

 Now let $w = s_{i_1}\ldots s_{i_{k}}s$ and $w = s_{j_1}\ldots s_{j_{k}}s$ be two reduced words for $w$, ending in $s$. By Lemma \ref{Lemma append same gen comparison}, $ s_{i_1}\ldots s_{i_{k}}s <_\Delta  s_{j_1}\ldots s_{j_{k}}s$ if and only if $ s_{i_1}\ldots s_{i_{k}} <_\Delta  s_{j_1}\ldots s_{j_{k}}$. Hence, the least reduced word for $w$ with $s$ as its right-most generator must be $\mathrm{NF}_\Delta(ws) = \mathrm{NF}_\texttt{Rlex}(ws)$. 

Now suppose we have distinct $s,r \in S$ such that $\ell(ws) = \ell(wr) = \ell(w)-1$. We compare the least words ending in each of these generators with respect to $<_\Delta$. By Lemma 2.2.1 of \cite{COCG}, we know that there are reduced words for $ws$ and $wr$ that are respectively subwords of the reduced expressions $\mathrm{NF}_\Delta(wr)r$ and $\mathrm{NF}_\Delta(ws)s$. Call these words for $w$, $w_s^*$ and $w_r^*$, say. Now Lemma \ref{LemmaSubwords} tells us that $\mathrm{NF}_\Delta(ws) \le_\Delta w_s^* <_\Delta  \mathrm{NF}_\Delta(wr)r$ and $\mathrm{NF}_\Delta(wr) \le_\Delta w_r^* <_\Delta  \mathrm{NF}_\Delta(ws)s$. More trivially, we also have $\mathrm{NF}_\Delta(ws) <_\Delta  \mathrm{NF}_\Delta(ws)s$ and $\mathrm{NF}_\Delta(wr) <_\Delta  \mathrm{NF}_\Delta(wr)r$. So we can apply Lemma \ref{Lemma split word on right}, whence $\mathrm{NF}_\Delta(wr)r <_\Delta  \mathrm{NF}_\Delta(ws)s$ if and only if $r<s$. Therefore $\mathrm{NF}_\Delta(w)$ is $\mathrm{NF}_\Delta(wr)r$ for the least possible such $r \in S$.
\end{proof}

A key property of Coxeter groups is that for all $w \in W$, $w$ can be uniquely decomposed into the form
$w = w_n\ldots w_1$ where each $w_i$ is a minimal (left) coset representative of $ W_{\{s_1,\ldots,s_{i-1}\}}$ in $ W_{\{s_1,\ldots,s_{i-1},s_i\}}$. Equivalently,  each $w_i$ is either the identity element, or is some element of $W_{\{s_1, \ldots, s_i\}}$ for which all of its reduced words end in $s_i$. Moreover,
$\mathrm{NF}_\texttt{Rlex}(w) = \mathrm{NF}_\texttt{Rlex}(w_n)\mathrm{NF}_\texttt{Rlex}(w_{n-1})\ldots \mathrm{NF}_\texttt{Rlex}(w_{1})$. See Proposition 3.4.2 of \cite{COCG} for details, noting that here we are considering the \textit{reverse} lexicographic order. Thus, using the language of Section 3,  $\delta_n(\mathrm{NF}_\texttt{Rlex}(w)) = \mathrm{NF}_\texttt{Rlex}(w_n)$ and $\tau_n(\mathrm{NF}_\texttt{Rlex}(w)) = \mathrm{NF}_\texttt{Rlex}(w_{n-1})\ldots \mathrm{NF}_\texttt{Rlex}(w_{1})$. We abuse our notation to write $\delta_i(w) = w_i$ in this context. 

\begin{corollary}\label{corollary Heads and Tails Cox general}
For all $u,v \in W$,
and $j = 1,\ldots, n$, 
$ u <_\Delta^W v$ if and only if
\begin{enumerate}[$(i)$]
    \item $u_n <_\Delta^W  v_n$; or
    \item $u_{n} = v_n$ and $u_{j-1}\ldots,u_{1} <_\Delta  v_{j-1}\ldots,v_{1}$.
\end{enumerate}
\end{corollary}
\begin{proof}
    This is a direct corollary of Lemma \ref{Lemma delta dichotomy}.
\end{proof}

Before proceeding, we recall the definition of the Bruhat order on $W$ and some of its basic properties. 
We refer the reader to Section 2 of \cite{COCG} for a thorough introduction on the Bruhat order.
Let $T=S^W$ and $u,v \in W$, then $u <_B v$ if there exists a sequence $t_1, \ldots, t_k \in T$ such that $v = ut_1\ldots t_k$ and 
$\ell(u) < \ell(ut_1) < \ldots < \ell(ut_1\ldots t_k) = \ell(v)$.
More relevant to us, is the characterisation of the Bruhat order by Lemma 2.2.1. of \cite{COCG}: $u <_B v$ if and only if for any reduced word for $v$, $v^*$, there exists some reduced word for $u$, $u^*$, that is a subword of $v^*$.

\begin{corollary}\label{Lemma Bruhat Less than}
Let $u,v \in W$ be such that $u <_B v$. Then $u <_\Delta^W v$.
\end{corollary}
\begin{proof}
Since $u <_B v$, we can fix reduced words for $u$ and $v$ respectively, $u^*$ and $v^*$ such that letting $v^* = \mathrm{NF}_\Delta(v)$ and using the fact that $\mathrm{NF}_\Delta(u) \le_\Delta u^*$, by Lemma \ref{LemmaSubwords} we know $\mathrm{NF}_\Delta(u) \le_\Delta u^* <_\Delta  \mathrm{NF}_\Delta(v)$, whence $u <_\Delta^W v$.
\end{proof}

\section{Successor Algorithm}\label{algorithm}

Let the Cayley graph of $(W, S)$ be denoted by $\mathrm{Cay}(W, S)$. In this context, we consider there to be a unique, undirected, $s$-labelled edge in the graph between the elements $w$ and $ws$, for all $w \in W$ and $s \in S$. We draw attention to the fact that we are multiplying our elements of $S$ on the right of $w$.

We present the following greedy algorithm on $\mathrm{Cay}(W, S)$ in the case when $W$ is finite.
\begin{algorithm} \label{Algorithm min min}
{The algorithm takes two inputs}: The Cayley graph $\mathrm{Cay}( W, S )$ and a choice of total order on $S$, $<_S$.
It outputs the following: $L$, a labelling (bijection) from $W$ to the set $\{ 1, ..., | W | \}$ and $T'$, a spanning tree of $\mathrm{Cay}( W, S )$.
        \par 
        
    The algorithm proceeds as follows: First, set $L( id ) =  1$ where id denotes the identity element of $W$. Then, for $k > 1$, repeat the following step until $T'$ is a spanning tree of $\mathrm{Cay}( W, S )$:
        \begin{itemize}
            \item Set $X = \{ w \in W \,|\, L( w ) < k \}$;
            \item Set $Y = \{ s \in S \,|\, \text{ there exists } w \in X \text{ such that } ws \notin X \}$;
            \item Set $y = min( Y )$ with respect to $<_S$;
            \item Set $x = min( X )$ with respect to $L$ such that $xy \notin X$;
            \item Set $L( x ) = k $;
            \item Append the edge $\{ x, xy \}$ to $T'$.
        \end{itemize}
\end{algorithm}

Algorithm \ref{Algorithm min min} can be simplified to a more concise form: greedily choose the next element to add to ${X}$ by first considering those edges in $\mathrm{Cay}(W, S)$ connecting elements in $X$ to elements in $W\setminus X$ of minimal $S$-label, and secondly, minimise the $L$ value of the nodes of these edges in $X$. This can be seen as an instance of Prim's greedy algorithm (\cite{prims}) on $\mathrm{Cay}(W, S)$: take the initial vertex to be the identity element, each $s_i$-edge is weighted $i$, and the tie-breaking between edges is resolved by choosing that edge with least, already labelled vertex. Since Prim's algorithm produces a minimal spanning tree, we know that ${T'}$ is minimal.

We present some small examples by way of illustration.

\begin{example}\label{Example Prims sym3}
    We consider the example $(W,S) = \Sym{3}$ generated by $s_1 = (1, 2)$ and $s_2 = (2, 3)$, so $s_1 <_S s_2$. We associate the black edges with edges labelled $s_1$ in $\mathrm{Cay}(W,S)$ and grey for $s_2$.
\begin{figure}[H]\label{Figure sym3 example}
    \centering
    % First subfigure
    \begin{subfigure}{0.45\textwidth} % Adjust width as needed
        \centering
        \includegraphics[width=6.5cm]{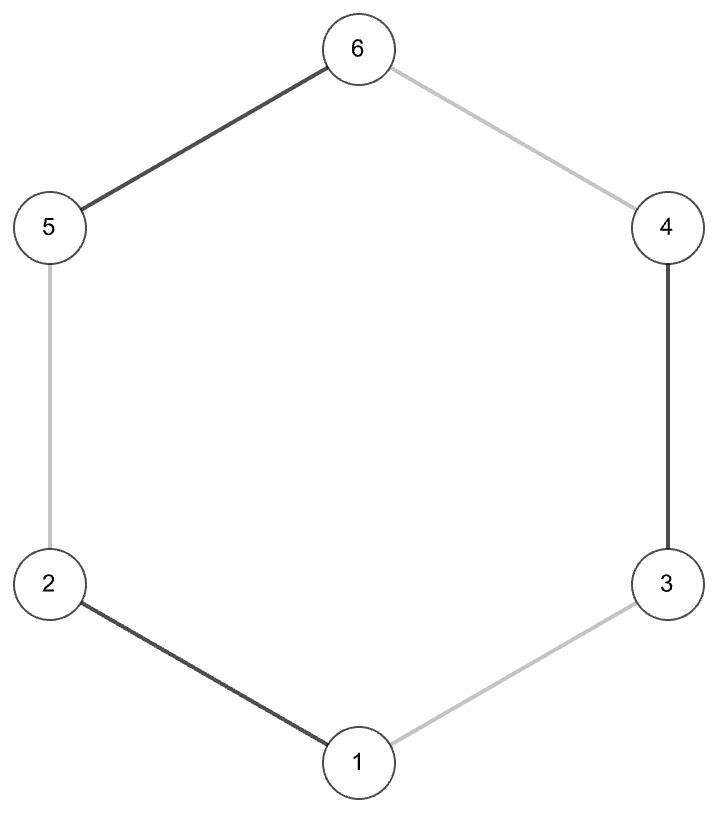} % Replace with your image
        \label{fig:subfig1}
    \end{subfigure}
    \hfill % Adds space between subfigures
    % Second subfigure
    \begin{subfigure}{0.45\textwidth} % Adjust width as needed
        \centering
        \includegraphics[width=6.5cm]{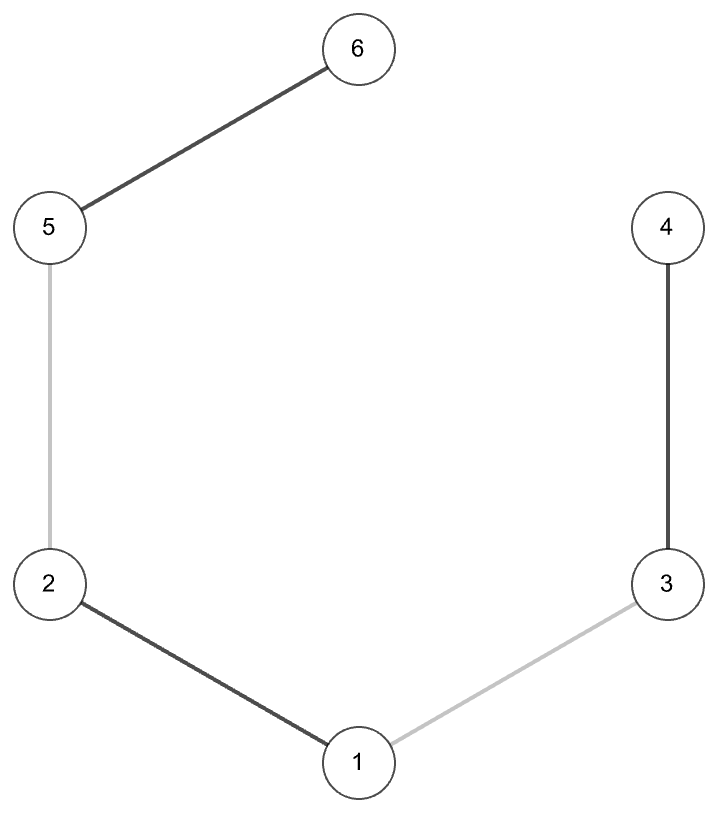} % Replace with your image
        \label{fig:subfig2}
    \end{subfigure}
    
    \caption{The Cayley Graph of $\Sym{3}$ with the labelling $L$ produced by Algorithm \ref{Algorithm min min} (left) and the resulting minimal spanning tree (right).}
    \label{fig:mainfig}
\end{figure}

\begin{table}[H]
    \centering

    \begin{tabular}{r|c|l}
 1&$ 1\, 2\, 3  $&$ \mathrm{e} $ \\  
 2&$ 2\, 1\, 3  $&$ s_1 $ \\  
 3&$ 1\, 3\, 2 $&$ s_2 $ \\  
 4&$ 2\, 3\, 1 $&$ s_2s_1 $ \\  
 5&$ 3\, 1\, 2 $&$ s_1s_2 $ \\  
 6&$ 3\, 2\, 1 $&$ s_1s_2s_1 $
    \end{tabular}
    \caption{The $<_\Delta$-ordering on $\Sym{3}$. From left to right for each $w \in W$: $L(w)$, its image as a permutation, $(1)w, (2)w, (3)w$, and $\mathrm{NF}_\Delta(w)$ . }
    \label{Table S4 ShortLoc}
\end{table}

Next, we present the example  $(W,S) = \Sym{4}$ generated by $s_1 = (1, 2)$, $s_2 = (2, 3)$ and $s_3 = (3, 4)$, and $<_S$ given by $s_1 <_S s_2 <_S s_3$. In this case, edges labelled $s_1$, $s_2$ and $s_3$ correspond to the black, grey and dashed edges.

\begin{figure}[H]
    \centering
    % First subfigure
    \begin{subfigure}{1\textwidth} % Adjust width as needed
        \centering
        \includegraphics[width=10cm]{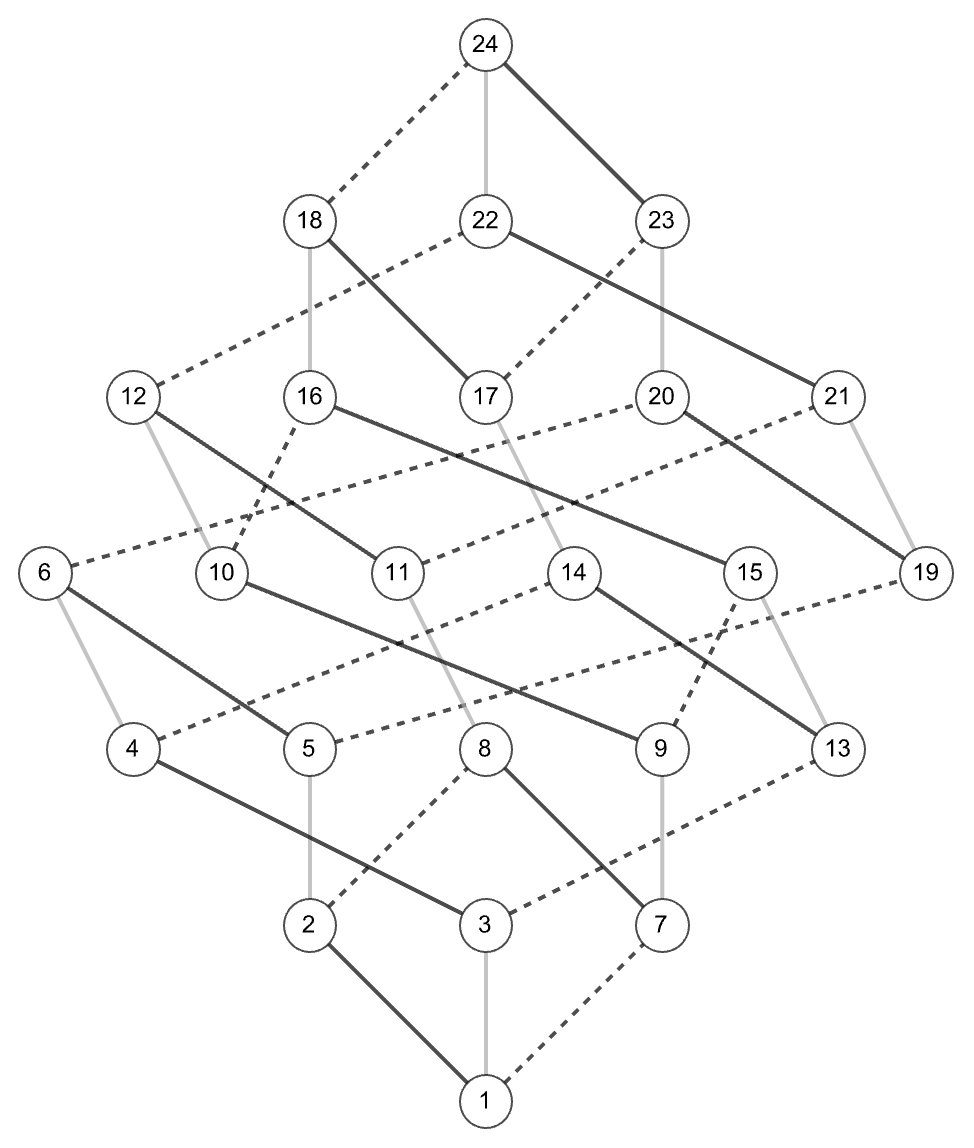} % Replace with your image
        \label{fig:subfig1}
    \end{subfigure}
    
    \caption{The Cayley Graph of $\Sym{4}$ with the labelling $L$ produced by Algorithm \ref{Algorithm min min}.}
    \label{fig:mainfig}
\end{figure}

\begin{figure}[H]
    \centering
    % First subfigure
    \begin{subfigure}{1\textwidth} % Adjust width as needed
        \centering
        \includegraphics[width=10cm]{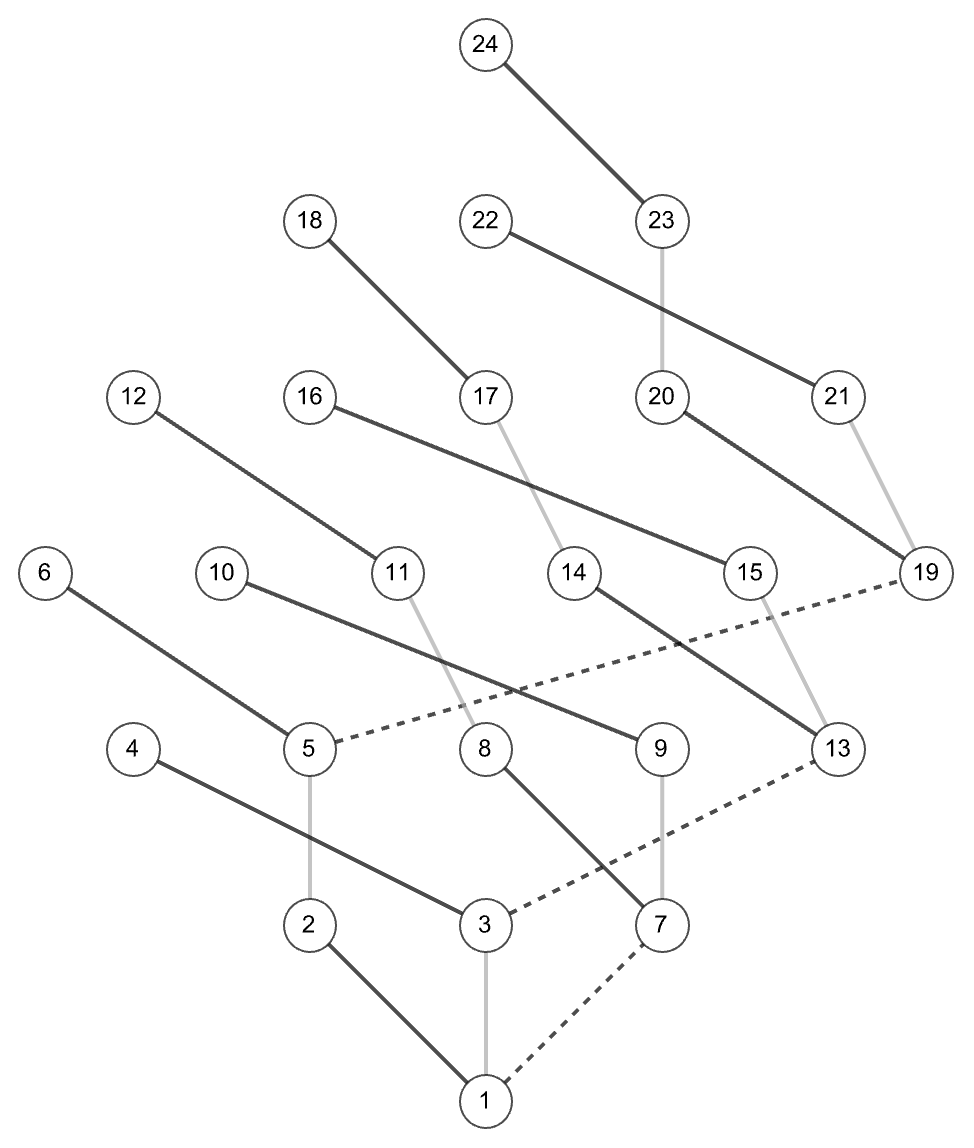} % Replace with your image
        \label{fig:subfig1}
    \end{subfigure}
    
    \caption{The minimal spanning tree of $\Sym{4}$ with the labelling $L$ produced by Algorithm \ref{Algorithm min min}.}
    \label{fig:mainfig}
\end{figure}

\begin{table}[H]
    \centering
\begin{tabular}{c c}

    \begin{tabular}{r|c|l}
  1&$ 1\, 2\, 3\, 4\,  $&$  $ \\  
  2&$ 2\, 1\, 3\, 4\,  $&$ s_1 $ \\  
  3&$ 1\, 3\, 2\, 4\,  $&$ s_2 $ \\  
  4&$ 2\, 3\, 1\, 4\,  $&$ s_2s_1 $ \\  
  5&$ 3\, 1\, 2\, 4\,  $&$ s_1s_2 $ \\  
  6&$ 3\, 2\, 1\, 4\,  $&$ s_1s_2s_1 $ \\  

   7&$ 1\, 2\, 4\, 3\, $&$ s_3 $ \\  
   8&$ 2\, 1\, 4\, 3\, $&$ s_3s_1 $ \\  
   9&$ 1\, 3\, 4\, 2\, $&$ s_3s_2 $ \\  
  10&$ 2\, 3\, 4\, 1\,  $&$ s_3s_2s_1 $ \\  
  11&$ 3\, 1\, 4\, 2\,  $&$ s_3s_1s_2 $ \\  
  12&$ 3\, 2\, 4\, 1\,  $&$ s_3s_1s_2s_1 $ 
    \end{tabular}&\hspace{2cm}
    \begin{tabular}{r|c|l}
 13&$ 1\, 4\, 2\, 3\,  $&$ s_2s_3 $ \\  
 14&$ 2\, 4\, 1\, 3\,  $&$ s_2s_3s_1 $ \\  
 15&$ 1\, 4\, 3\, 2\,  $&$ s_2s_3s_2 $ \\  
 16&$ 2\, 4\, 3\, 1\,  $&$ s_2s_3s_2s_1 $ \\  
 17&$ 3\, 4\, 1\, 2\,  $&$ s_2s_3s_1s_2 $ \\  
 18&$ 3\, 4\, 2\, 1\,  $&$ s_2s_3s_1s_2s_1 $ \\ 

 19&$  4\, 1\, 2\, 3\,  $&$ s_1s_2s_3 $ \\  
 20&$  4\, 2\, 1\, 3\,  $&$ s_1s_2s_3s_1 $ \\  
 21&$  4\, 1\, 3\, 2\,  $&$ s_1s_2s_3s_2 $ \\  
 22&$  4\, 2\, 3\, 1\,  $&$ s_1s_2s_3s_2s_1 $ \\  
 23&$  4\, 3\, 1\, 2\,  $&$ s_1s_2s_3s_1s_2 $ \\  
 24&$  4\, 3\, 2\, 1\,  $&$ s_1s_2s_3s_1s_2s_1 $ 
    \end{tabular}
    \end{tabular}
    \caption{The $<_\Delta$-ordering on $\Sym{4}$. From left to right for each $w \in W$: $L(w)$, its image as a permutation, $(1)w, (2)w, (3)w, (4)w$, and $\mathrm{NF}_\Delta(w)$ . }
    \label{Table S4 ShortLoc}
\end{table}

\end{example}

In our next theorem, we see that the labelling of $\mathrm{Cay}(W, S)$ by Algorithm \ref{Algorithm min min} and the deletion order on elements of $W$ coincide. 
\begin{theorem}\label{Theorem min min equal deletion}
Let $w \in W$ and let $L(w) = k$. Then $w$ is the $k^{th}$ element of $W$ with respect to $<_\Delta^W$.
\end{theorem}
\begin{proof}
We proceed by induction on $L(w)$. Note that $\mathrm{NF}_\Delta(id)$ is the empty word and consequently the minimum word with respect to $<_\Delta $. In  Step 1 in Algorithm \ref{Algorithm min min}, we defined $L(id) = 1$  and hence the base case is satisfied.

Now suppose for some $k\ge 1$, our statement holds for all $u \in W$ such that $L(u) \le k$. Let $w \in W$ be that element such that $L(w) = k+1$; the least element with respect $<^W_\Delta$ that is not labelled in the $k^{th}$ step of Algorithm \ref{Algorithm min min}. By Theorem \ref{Theorem DeltaNF = RLexNF}, we know that $\mathrm{NF}_\Delta(w) = \mathrm{NF}_\texttt{RLex}(w)$ and hence $\mathrm{NF}_\Delta(w) = \mathrm{NF}_\texttt{RLex}(ws)s$ for minimum $s \in S$ (with respect to $<_S$) such that $\ell(ws) = \ell(w) -1$. 

We also know that $w$ must be adjacent to some previously labelled vertex in $\mathrm{Cay}(W, S)$. Since, if not, then there exists some unlabelled element $v \in W$ such that $v <_B w$ whence Lemma \ref{LemmaSubwords} tells us that $L(v)<L(w)$, a contradiction. 

Let $x \in W$ be another element that is unlabelled and adjacent to some labelled vertex of $\mathrm{Cay}(W, S)$ in the $k^{th}$ step of Algorithm \ref{Algorithm min min}. And let $r$ be the least label of the edges connecting $x$ to a labelled vertex. Therefore $\mathrm{NF}_\Delta(x) = \mathrm{NF}_\texttt{RLex}(xr)r$. Since both $w$ and $x$ are unlabelled, we must have 

$$\mathrm{NF}_\texttt{RLex}(xr) <_\Delta  \mathrm{NF}_\texttt{RLex}(x),  \mathrm{NF}_\texttt{RLex}(xr) <_\Delta \mathrm{NF}_\texttt{RLex}(w),$$
$$\mathrm{NF}_\texttt{RLex}(ws) <_\Delta  \mathrm{NF}_\texttt{RLex}(x),  \mathrm{NF}_\texttt{RLex}(ws) <_\Delta \mathrm{NF}_\texttt{RLex}(w).$$

Hence we can apply Lemma \ref{Lemma split word on right} whence we know that $s$ must be of minimal possible edge label amongst those edges in $\mathrm{Cay}(W, S)$ connecting a labelled and unlabelled vertex and, amongst those edges of label $s$, $L(ws)$ must be minimal. Therefore $L(w)$ is indeed the label of the $(k+1)^{th}$ least element of $W$ with respect to $<^W_\Delta$.
\end{proof}

\begin{corollary}\label{Corollary min min equal deletion}
For all $w \in W$, $$L(w_1) < L(w_2) \text{ if and only if } w_1 <_\Delta^W w_2.$$
\end{corollary}

Recall that for $w \in W$, $L(w) \in \{1, \ldots, |W|\}$ and so $<$ in Corollary \ref{Corollary min min equal deletion} is just the usual inequality between natural numbers.

\section{Artinian Orders}\label{artinianorders}

In this section, we look at a finiteness condition on $<_{\Delta}^W$ which we now define.

\begin{definition}[Artinian Orders]\label{definition :artinianorders}
Let $(W, S)$ be a (not necessarily finite nor irreducible) Coxeter group equipped with a total ordering on $S$, $<_S$. Then we call $<_\Delta^W$ \textit{Artinian} if for all elements $w \in W$, the set $w^\downarrow := \{ u \in W | u <_\Delta^W w \}$ is finite.

When $<_\Delta^W$ is indeed {Artinian}, we define $L_0(w) = |w^\downarrow|$ for all $w \in W$.
\end{definition}

Note that $L_0(w) = L(w) - 1$ for all $w \in W$ for $L$ described in Algorithm \ref{Algorithm min min}.
Being Artinian is the exact condition needed for Algorithm \ref{Algorithm min min} to be able to label any element of $W$ after finitely many steps.
We now see that determining if $<_\Delta^W$ is Artinian is very easily verified.

\begin{lemma}\label{Lemma artinian iff finite}
The deletion order $<_\Delta^W$ on $(W,S)$ is Artinian if and only if $W_{\{s_1,\ldots,{s_{n-1}}\}}$  is finite.
\end{lemma}
\begin{proof}
Suppose $<_\Delta^W$ is Artinian. Then  $s_n^\downarrow$ must be finite. But for all $w \in W_{\{s_1,\ldots{s_{n-1}}\}}$, $\mathrm{NF}_\Delta(w)$ has no appearances of $s_n$. Hence $ W_{\{s_1,\ldots{s_{n-1}}\}} \subseteq s_n^\downarrow$. So $W_{\{s_1,\ldots{s_{n-1}}\}}$ must be finite.

Suppose $W_{\{s_1,\ldots,{s_{n-1}}\}}$ is finite. We show that for all $w \in W$, there are only finitely many possible $v \in w^\downarrow$. Write the deletion sequence of $v$, $$\widehat{\mathrm{NF}_\Delta(v)}_n = [v_0,\ldots, v_k].$$ Now the number of appearances of $s_n$ in $\mathrm{NF}_\Delta(v)$ must be less than or equal to that of $\mathrm{NF}_\Delta(w)$. Equivalently,  $k=  \lambda_n(\mathrm{NF}_\Delta(v)) \le \lambda_n(\mathrm{NF}_\Delta(w)) $. Each of the words $v_0,\ldots, v_k $ is the $<_{\Delta}$-normal form for some element in $W_{\{s_1,\ldots,{s_{n-1}}\}}$. This is in bijection with $W_{\{s_1,\ldots,{s_{n-1}}\}}$, a finite set. Hence the number of possibilities for the deletion sequence of $\widehat{\mathrm{NF}_\Delta(v)}_n$ is bounded above by $|W_{\{s_1,\ldots,{s_{n-1}}\}}|^{\lambda_n(w)+1}$ and therefore finite.
\end{proof}

An immediate consequence of this is that those irreducible Coxeter systems for which $<_\Delta^W$ is Artinian, regardless of choice of ordering $<_S$ on $S$, are exactly those whose proper parabolic subgroups are all finite. Such Coxeter groups are classified, leading to the following theorem.

\begin{theorem}\label{corollary Artinian compact hyperbolic}
Suppose that $W$ is a finitely generated irreducible Coxeter group. Then $<_\Delta^W$ is Artinian for all $<_S$ on $S$ if and only if $(W, S)$ is finite, affine, or compact hyperbolic.
\end{theorem}
\begin{proof}
It is shown in \cite{humphreys} that those irreducible Coxeter groups whose proper parabolic subgroups are all finite are exactly the finite, affine, or compact hyperbolic groups.
\end{proof}

\begin{corollary}\label{Lemma Decomposition Formula}
If $<_\Delta^W$ is Artinian,
then $$ L_0(w) = \sum_{i=1}^n L_0(\delta_n(w)). $$ 
\end{corollary}
\begin{proof}
Let $w\in W$ and write $w = w_n w_{n-1}\ldots w_1 $ in its unique decomposition of the form where each $w_i$ is a minimal left coset representative of $W_{\{s_1,\ldots, s_{i-1}\}}$ in $W_{\{s_1,\ldots, s_i\}}$. By Lemma \ref{Lemma artinian iff finite}, we know that since $<_\Delta^W$ is Artinian, $W_{\{s_1,\ldots, s_{n-1}\}}$ is finite. Therefore, the image, $L_0(W_{\{s_1,\ldots, s_{n-1}\}})$, is finite also and so $L_0(w_{n-1}\ldots w_1) \in \mathbb{N}$. Now applying Lemma \ref{Lemma delta dichotomy} the result follows by induction on $|S|$.
\end{proof}

Corollary \ref{Lemma Decomposition Formula} shows that to compute $L_0(w)$ for all $w \in W$, it suffices to compute the order only for the minimal coset representatives of $W_{\{s_1,\ldots, s_i\}}/W_{\{s_1,\ldots, s_{i-1}\}}$ for $i=1,\ldots,n$. This also hints at a more succinct, recursive reformulation of Algorithm \ref{Algorithm min min} that we do not explore here.

\section{Duality for Coxeter Groups}

In this section, we assume $W$ is a finite Coxeter group and recall that each finite Coxeter group contains a unique, longest element, $\omega_0$, with respect to $\ell$, which is an involution. Abusing our notation of Section \ref{section intro}, we let $\varphi: W \rightarrow \Sym{|W|}$ be the Cayley embedding induced by $<_\Delta$. That is, $(L(u))\varphi(w) = L(uw^{-1})$. If  $L(w) + L(\omega_0w) = |W| + 1$ for all $w \in W$, then we refer to this property as the \emph{duality property}.

We see that understanding exactly which orderings of have the duality property is non-trivial by showing it holds for two families but not for another. Theorem \ref{duality} looks at the deletion order when $W$ is either of type $\mathrm{A}_n$ or $\mathrm{B}_n$ and assumes their Coxeter diagrams are as follows,  with $s_i <_S s_j$ whenever $i<j$.
\\
\begin{center}

\begin{tikzpicture}[scale=0.5]

\draw (-2,0-4) -- (2,0-4);

\draw (6,0-4) -- (8,0-4);

\node (s1) at (-2,-0.6-4) {$s_1$};
\node (s1) at (0,-0.6-4) {$s_2$};

\node (s1) at (2,-0.6-4) {$s_3$};
\node (s1) at (6,-0.6-4) {$s_{n-1}$};
\node (s1) at (8,-0.6-4) {$s_n$};

\filldraw [black] (-2,-0.0-4) circle (5pt);
\filldraw [black] (0,-0.0-4) circle (5pt);
\filldraw [black] (2,-0-4) circle (5pt);
\filldraw [black] (6,0-4) circle (5pt);
\filldraw [black] (8,-0-4) circle (5pt);

\filldraw [black] (3.333,-0-4) circle (1pt);
\filldraw [black] (4.6666,-0-4) circle (1pt);
\filldraw [black] (4,-0-4) circle (1pt);

\draw[double,  double distance=.1cm] (-2,0-4-4) -- (0,0-4-4);
\draw (0,0-4-4) -- (2,0-4-4);

\draw (6,0-4-4) -- (8,0-4-4);

\node (s1) at (-2,-0.6-4-4) {$s_1$};
\node (s1) at (0,-0.6-4-4) {$s_2$};

\node (s1) at (2,-0.6-4-4) {$s_3$};
\node (s1) at (6,-0.6-4-4) {$s_{n-1}$};
\node (s1) at (8,-0.6-4-4) {$s_n$};

\filldraw [black] (-2,-0.0-4-4) circle (5pt);
\filldraw [black] (0,-0.0-4-4) circle (5pt);
\filldraw [black] (2,-0-4-4) circle (5pt);
\filldraw [black] (6,0-4-4) circle (5pt);
\filldraw [black] (8,-0-4-4) circle (5pt);

\filldraw [black] (3.333,-0-4-4) circle (1pt);
\filldraw [black] (4.6666,-0-4-4) circle (1pt);
\filldraw [black] (4,-0-4-4) circle (1pt);
\end{tikzpicture}
\end{center}

Recall that edges in $\mathrm{Cay}(W,S)$ are given by $(w,ws)$ where $w \in W,  s \in S$.

\begin{theorem}\label{duality} Suppose that $W$ is the Coxeter group of either type $\mathrm{A}_n$ or $\mathrm{B}_n$. Then for each $w \in W$
$$L(w) + L(\omega_0w) = |W| + 1.$$
\end{theorem}

\begin{proof} We argue by induction on $n$. If $n = 1$, then $W = \langle s_1 \rangle$ with $L(1) = 1, L(s_1) =2$ and $\omega_0 = s_1$, and the theorem clearly holds. Now assume $n > 1$ and put $X = \langle s_1,s_2, \dots, s_{n-1} \rangle$. Let $\omega_0^X$ be the longest element of $X$. Observe that the deletion order on $\mathrm{Cay}(X, \{s_1, \dots,s_{n-1} \})$ coincides with the restriction of the deletion order on $\mathrm{Cay}(W,S)$ to this subgraph. In particular, using $L_X$ for positions in $X$, we have $L(w) = L_X(w)$ for all $w \in X$. Hence, by induction, $L_X(1) + L_X(\omega_0^X) = |X| + 1$, and so $L(\omega_0^X) = |X|$.

Put 
$$\mathcal{C} = \{1, s_n, s_{n-1}s_n,s_{n-2}s_{n-1}s_n, \dots, s_1s_2\cdots s_{n-1}s_n \}$$
if $W$ is of type $\mathrm{A}_n$ and 
$$\mathcal{C} = \{1, s_n, s_{n-1}s_n, \dots, s_1s_2 \cdots s_n,  s_2s_1s_2 \cdots s_n, \dots,s_ns_{n-1}\cdots s_2s_1s_2 \cdots s_{n-1}s_n \}$$
if $W$ is of type $\mathrm{B}_n$. Then $\mathcal{C}$ is a set of left coset representatives for $X$ in $W$ with the property that they are each of minimal length in their $X$ left coset (see 2.2.4 of \cite{GeckPfeiffer}).\\

(\ref{duality}.1) Let $w \in W$ with $w = cx$ where $c \in \mathcal{C}$ and $x \in X$. Then 
$$L(w) + 1 = L(c) + L(x).$$

By Corollary \ref{Lemma Decomposition Formula}
$$L_0(w) = L_0(\delta_n(w)) + \sum_{i=1}^{n-1} L_0(\delta_i(w)).$$
Since $c$ ends in $s_n$ and $x$ has no appearances of $s_n$, $\delta_n(w) = c$ and $x = \delta_{n-1}(w)\cdots \delta_1(w)$. Using Corollary \ref{Lemma Decomposition Formula} again yields
$$L_0(w) = L_0(c) + L_0(x).$$
Therefore 
$$L(w) - 1 = L(c) -1 + L(x) -1 ,$$
giving (\ref{duality}.1).\\

(\ref{duality}.2) Let $c \in \mathcal{C}$. Then $L(c) = \ell(c)|X| + 1$.\\

By Lemma \ref{LemmaSubwords} we have 
$$L(1) < L(s_n) < L(s_{n-1}s_n) < \dots < L(s_1 \dots s_n)$$
in the case $W$ is of type $\mathrm{A}_n$ and 
$$L(1) < L(s_n) < L(s_{n-1}s_n) < \dots < L(s_1 \cdots s_n) < L(s_2s_1s_2 \dots s_n) $$ $$< \dots < L(s_ns_{n-1} \cdots s_2s_1s_2 \cdots s_{n-1}s_n)$$
in the case when $W$ is of type $\mathrm{B}_n$. Since $s_i < s_n$ for all $i \ne n$, the definition of the deletion order means that the elements of $X$ will be labelled $1 = L(1)$ to $|X| = L(\omega^X_0)$. Moreover, by (\ref{duality}.1), the elements in $\mathcal{C}$ will have the least label in each of their $X$-left cosets. So $s_n$ will have the least label in $W \setminus X$. Therefore $L(s_n) = |X| + 1$. Then using (\ref{duality}.1) $L(s_n\omega_0^X) = 2|X|$, the largest label in $X \cup s_nX$. Now $s_{n-1}s_n$ will have the least label in $W \setminus (X \cup s_nX)$ which gives $L(s_{n-1}s_n) = 2|X| + 1$. Continuing in this fashion yields (\ref{duality}.2).\\

Combining (\ref{duality}.1) and (\ref{duality}.2) gives \\

(\ref{duality}.3) If $w = cx \in W$ with $c \in \mathcal{C}$ and $x \in X$, then
$$L(w) = \ell(c)|X| + L(x).$$

% TODO: Change this.
(\ref{duality}.4) The left multiplication action by $\omega_0$ of the left cosets of $X$ interchanges $c_1X$ and $c_2X$ where $c_1, c_2 \in \mathcal{C}$ are such that $\ell(c_1) + \ell(c_2) = n$, when $W$ is of type $\mathrm{A}_n,$ and $\ell(c_1) + \ell(c_2) = 2n - 1$, when $W$ is of type $\mathrm{B}_n$.\\

By the length reversing property of $\omega_0$ (see 1.8 of \cite{humphreys}) $\omega_0$ must interchange $c_1$ and $c_2\omega^X$ and interchange $c_2$ and $c_1\omega_0^X$. This yields (\ref{duality}.4).\\

(\ref{duality}.5) Let $w \in W$. If $s_{i_1} \cdots s_{i_t}$ is a minimal path in $\mathrm{Cay}(W,S)$ such that $ws_{i_1} \cdots s_{i_t} = \omega_0$, then $\omega _0w = s_{i_t} \cdots s_{i_1}$.\\

This follows easily as the elements of $S$ are involutions.\\

We note that this means $\omega_0w$ can be located in $\mathrm{Cay}(W,S)$ 
by starting at $1$ and following the edges labelled, in order, $s_{i_t}$, then 
$s_{i_{t-1}}$ and so on ending with $s_{i_1}$.

So now letting $w \in c_1X$ and $\omega_0w \in c_2X$ we show that $L(w) + L(\omega_0 w) = |W| + 1$. Let $u_1, u_2, u_3$ be minimal paths in $\mathrm{Cay}(W,S)$ such that $c_1u_1 = w, wu_2 = c_1\omega_0^X$ (= the longest element in $c_1X$) and $c_1\omega_0^Xu_3 = w$. Observe that, as elements of $W$, $u_1, u_2 \in X$. Also by (\ref{duality}.5) $\omega_0w = u_3^{-1}u_2^{-1}$. Since $\omega_0 c_1\omega_0^X = c_2$ we have $$\omega_0 w = c_2u_2^{-1}.$$ Further $w = c_1u_1$. Now looking in $X$ and using (\ref{duality}.5)  there,  $\omega_0^Xu_1 = u_2^{-1}.$ By induction
$$L(u_1) + L(u_2^{-1}) = L(u_1) + L(\omega_0^Xu_1) = |X| + 1.$$
We have 
$$L(w) + L(\omega_0 w) = L(c_1u_1) + L(c_2u_2^{-1}),$$
and using (\ref{duality}.1) this gives

$$L(w) + L(\omega_0 w) = L(c_1) + L(u_1) + L(c_2) + L(u_2^{-1}) - 2.$$
When $W$ is of type $\mathrm{A}_n$, employing (\ref{duality}.2) and (\ref{duality}.4) yields
\begin{align*}
L(w) + L(\omega_0 w)&= \ell(c_1)|X| +L(u_1) + \ell(c_2)|X| + L(u_2^{-1})\\
&=\ell(c_1)|X| + (n - \ell(c_1))|X| + L(u_1) + L(u_2^{-1})\\
&= n|X| + |X| + 1\\
&= (n + 1)|X| + 1 = |W| +1,\\
\end{align*}
which establishes the theorem in that case. While when $W$ is of type $\mathrm{B}_n$
\begin{align*}
L(w) + L(\omega_0 w) & = \ell(c_1)|X| + (2n - 1 - \ell(c_1))|X| + |X| + 1\\
&= (2n - 1)|X| + |X| + 1\\
&=2n|X| +1 = |W| + 1,\\
\end{align*}
and completes the proof of Theorem \ref{duality}

\end{proof}

The action of $\omega_0$ on the left cosets of $X$ in the proof of Theorem \ref{duality} is clearly important and is, in fact, implicated in the failure of the duality property for other Coxeter groups. We see an example of this  for some ordering on type D Coxeter groups.

\begin{example}
Let $W = \langle s_i | i = 1,\ldots, 5\rangle$ be the Coxeter group of type $\mathrm{D}_5$ with Coxeter Diagram labelled

\begin{center}

\begin{tikzpicture}
\draw (2,0-4-4) -- (6,0-4-4);
\draw (2,0-4-4) -- (0,1-4-4);
\draw (2,0-4-4) -- (0,-1-4-4);

\node (s1) at (0,1-4-4-0.6) {$s_1$};
\node (s1) at (0,-1-4-4-0.6) {$s_3$};
\node (s1) at (2,-0.6-4-4) {$s_2$};
\node (s1) at (4,-0.6-4-4) {$s_4$};
\node (s1) at (6,-0.6-4-4) {$s_5$};

\filldraw [black] (0,1-4-4) circle (5pt);
\filldraw [black] (0,-1-4-4) circle (5pt);
\filldraw [black] (2,-0-4-4) circle (5pt);
\filldraw [black] (4,-0-4-4) circle (5pt);
\filldraw [black] (6,-0-4-4) circle (5pt);

\end{tikzpicture}
\end{center}
Again we employ the total order $<$ on the alphabet $S = \{s_i | i =1,\ldots,5\}$ where $s_i < s_j$ if and only if $i<j$. Put $X = \langle s_1,s_2,s_3,s_4 \rangle$. Then 

$$\mathcal{C} = \{1,s_5,s_4s_5,s_2s_4s_5,s_1s_2s_4s_5,s_3s_2s_4s_5,s_1s_3s_2s_4s_5,s_2s_1s_3s_2s_4s_5,$$ $$s_4s_2s_1s_3s_2s_4s_5,s_5s_4s_2s_1s_3s_2s_4s_5\} $$
is the set of minimal length left coset representatives for $X$ in $W$. Now $\omega_0$ fixes both of $s_1s_2s_4s_5X$ and $s_3s_2s_4s_5X$. Hence, by length considerations, $\omega_0s_1s_2s_4s_5$ is the longest element in $s_1s_2s_4s_5X$. Therefore 

\begin{align*}
L(\omega_0s_1s_2s_4s_5) &= 
L(s_1s_2s_4s_5) + |X|-1 \\
&= 4|X|+1+|X|-1 \\
&= 5|X|.\\
\intertext{Consequently}
L(s_1s_2s_4s_5) + 
L(\omega_0s_1s_2s_4s_5) &= 4|X|+1+5|X| \\
&= 9|X| + 1,
\end{align*}
whereas $|W|+1 = 10|X|+1.$
\end{example}

Finally we observe that the duality property holds for dihedral groups.  Let $W= \mathrm{I}_2(m)$
with $S = \{ s_1,s_2 \}$ and $s_1 < s_2$. Following the line of argument as in Theorem \ref{duality}, examining the left cosets of $\langle s_1 \rangle$ we determine $L(w), w \in W$. From this we then observe that the duality property holds.

\end{document}